\documentclass[12pt,arxiv, reqno]{ml-gen}

\usepackage{upref,mathrsfs,enumerate,bm} 

\usepackage[colorlinks=true,linkcolor=blue,citecolor=blue,backref=page]{hyperref}
\usepackage{local}

\opt{draft}{
\usepackage{showlabels}

\usepackage[displaymath,mathlines]{lineno}
\linenumbers
}

\usepackage{euler}

\begin{document}
\title[Spectral flow, Dirac operators, local boundary conditions]{%
On the spectral flow for Dirac operators with local boundary conditions}

\author{Alexander Gorokhovsky}
\address{Department of Mathematics, UCB 395,
University of Colorado, Boulder,
CO 80309-0395, USA}
\email{Alexander.Gorokhovsky@colorado.edu}
\thanks{Partially supported by the NSF grant DMS-0900968}

\author{Matthias Lesch}
\address{Mathematisches Institut,
Universit\"at Bonn,
Endenicher Allee 60,
53115 Bonn,
Germany}
\email{ml@matthiaslesch.de, lesch@math.uni-bonn.de}
\urladdr{www.matthiaslesch.de, www.math.uni-bonn.de/people/lesch}
\thanks{Partially supported by the
        Hausdorff Center for Mathematics}

\opt{draft,submit}{
\date{\today}
}
\begin{abstract}
Let $M$ be an even dimensional compact Riemannian manifold with boundary and
let $D$ be a Dirac operator acting on the sections of the Clifford module
$\cE$ over $M$.  We impose certain \emph{local} elliptic boundary conditions
for $D$ obtaining a selfadjoint extension $D_F$ of $D$. For a smooth
$U(n)$--valued function $g:M \to U(n)$ we establish a formula for the spectral
flow along the straight line between $D_F$ and $g\ii D_F g$. This spectral
flow is motivated by index theory: in odd dimensions it gives the natural
pairing between the $K$--homology class of the operator and the $K$--theory
class of $g$.

In our situation, with $\dim M$ having the ``wrong'' parity, the answer can be
expressed in terms of the natural spectral flow pairing on the
odd--dimensional boundary.

Our result generalizes a recent paper by M. Prokhorova \cite{pro} in which the
two-dimensional case is treated. Furthermore, our paper may be seen as an
odd-dimensional analogue of a paper by D. Freed \cite{fr}. As an application
we obtain a new proof of the cobordism invariance of the spectral flow.
\end{abstract}
\maketitle
\tableofcontents
%
\section{Introduction}\label{SecIntro}

Let $M$ be a compact even dimensional manifold with boundary. Let $D$ be a
Dirac operator acting on the sections of the Clifford module $\cE$ over $M$.
It is well-known that one can impose \emph{local} elliptic selfadjoint
boundary conditions for $D$. The boundary condition will be labeled by a
selfadjoint endomorphism $F \in \End_{\C(\pl M)}(\cE^0|_{\pl M})$ and the
corresponding realization of $D$ will be denoted by $D_F$.  Corresponding heat
kernel expansions have been considered in \cite{bragil}, and such boundary
conditions for Dirac operators were recently considered in \cite{cocham} to
construct the spectral action on a manifold with boundary.

Unlike the case of a closed manifold, $D_F$ is no longer an odd operator with
respect to the natural $\Z_2$--grading on $\cE$ and therefore one does not
have a natural index problem for this operator.

However, since the operator obtained is selfadjoint, there is a natural
spectral flow problem for this operator.  Namely, given a smooth function $g$
on $M$ with values in the unitary group $U(n)$ one can consider the spectral
flow between the operators $D\otimes \id$ and $g^{-1} (D \otimes \id) g$
acting on $L^2(M, \cE) \otimes \mathbb{C}^n$.  Equivalently, this spectral
flow equals the Toeplitz index $\ind(PgP)$, where $P$ is the positive spectral
projection of $D_F$ \cite{Baum82}, \cite[Sec. 3]{L}. For the case of
$2$-dimensional $M$ the spectral flow problem has been considered by
M.~Prokhorova in \cite{pro}. In this paper it is shown that the spectral flow
can be computed in terms of the winding numbers of $g$ at the boundary
components of $M$. In the paper \cite{kn} it was observed that one can reduce
the spectral flow problem to the index problem for the suspension of the
family. The latter index can be computed by the Atiyah-Bott index theorem
\cite{ab}.

The goal of the present paper is to extend the results of \cite{pro} to the
case of  arbitrary Dirac operators on compact even dimensional manifolds with
boundary.  Our main result Theorem \ref{ThmSF1} expresses the spectral flow of
the Dirac operator with local boundary conditions in terms of the spectral
flow of the boundary Dirac operator. As a consequence we obtain yet another
proof of the cobordism invariance of spectral flow. We note that even though
all the results are stated for Dirac operators, both the statements and the
proofs remain true for the more general first-order elliptic formally
self-adjoint operators.

In the even case, for the index of Dirac operators on odd-dimensional
manifolds, the corresponding result has been proved by D. Freed \cite{fr}.
Thus our paper can be considered as an odd analogue of one of D. Freed's
theorems. We note that the paper \cite{z} gives a heat equation proof of D.
Freed's theorem.

In the paper \cite{dz} a formula for the Toeplitz index for Dirac operators on
odd dimensional manifolds with global Atiyah-Patodi-Singer boundary conditions
is given. This problem is significantly different from the one we consider
here.

Our solution to the problem is based on the heat equation approach. We use the
heat equation expression for the spectral flow due to E.~Getzler \cite{G}.
Analytical formulas for spectral flow have received considerable attention in
the literature, see \cite{cp1, cp2, Carey2009}.

The paper is organized as follows. In section 2 we recall the main facts about
the spectral flow and quote Getzler's formula. Then, in section 3, we
introduce a class of local elliptic boundary value problems for Dirac
operators, state the spectral flow problem, our main result and a couple of
consequences. Finally, section 4 contains some heat kernel calculations on the
model cylinder $N\times [0,\infty)$. These are needed in the proof of Theorem
\ref{ThmSF1} to identify the ingredients in Getzler's formula in terms of
boundary data.

The authors are grateful to M.~Prokhorova for telling us about the results of
\cite{pro} and for many useful discussions. A.~Gorokhovsky is grateful to the
Hausdorff Center for Mathematics which supported his visit to Bonn. Finally,
we wish to thank the anonymous referee for thorough reading and helpful
comments.       
\section{Preliminaries}\label{SecPrelim}

For all notations and definitions related to Dirac operators and boundary
value problems our standard references are the books \cite{bgv}, \cite{bw}
and also the recent publication \cite{blz}.

\subsection{Spectral flow}\label{SSecSF}
Since the notion of spectral flow is central for this paper we recall here,
for the convenience of the reader, its definition.  Let $H$ be a separable
Hilbert space and denote by $\cf$ the set of closed self-adjoint Fredholm
operators equipped with the graph metric. For $T_1$, $T_2\in \cf$, the graph
distance is defined as
\begin{equation*}
    \|P_1 -P_2\|,
\end{equation*}
where $P_j$ is the projection onto the graph of $T_j$. An equivalent metric is
defined by
\begin{equation*}
\| (T_1+i)^{-1} - (T_2+i)^{-1} \|,
\end{equation*}
see \cite[Sec. 2]{L} for details.

Given a continuous path
\begin{equation*}
    f \colon[0, 1] \to \cf
\end{equation*}
one chooses a subdivision $0=t_0 <t_1 < \ldots <t_n=1$  of the interval
$[0,1]$ such that there exist $\varepsilon_j >0$, $j=1, \ldots, n$, satisfying
$\pm \varepsilon_j \notin \spec f(t)$ and
$[-\varepsilon_j, \varepsilon_j] \cap \specess f(t) = \emptyset$
for $t_{i-1} \le t \le t_j$. Then the \emph{spectral flow} of $f$ is defined by
\begin{equation}\label{DefSF}
\SF(f):= \sum_{j=1}^n \Bigl(\rank\bigl(\1_{[0,\varepsilon_j)}(f(t_j))\bigr)
  -\rank\bigl(\1_{[0, \varepsilon_j)}(f(t_{j-1}))\bigr)\Bigr).
\end{equation}
Here $\1_{[0, \varepsilon)}$ denotes the characteristic function of the
interval $[0, \varepsilon)$ and the operator $\1_{[0, \varepsilon_j)}(f(t_j))$ 
is defined via the functional calculus and therefore is equal to the
orthogonal projection onto the sum of the eigenspaces of $f(t_j)$
corresponding to the eigenvalues $\lambda \in [0, \varepsilon_j)$.  
It can be shown that the spectral flow thus defined depends only on the path
$f(t)$ and not on the choices made in the construction.

This definition essentially goes back to \cite{aps}. The approach we discuss
here appeared in the case of bounded operators  in \cite{P}, and in \cite{blp}
in the unbounded case. For a detailed historical discussion see
\cite[Introduction]{L}.

We emphasize that in this definition of $\SF$ it is not necessary that $f(0)$
and $f(1)$ are invertible.  It should be noted, however, that in the
literature there exist different conventions for dealing with the degenerate
case in which one or both endpoint values are not invertible.

Consider now a selfadjoint Fredholm operator $D \in \cf$. Let $g \in U(H)$ be
a unitary operator such that
\smallskip
\begin{enumerate}[\textup (i)]\topsep5em\itemsep0.4em
 \item\label{i} $g$ preserves $\Dom (D)$, the domain of
  $D$.\eqnnumlbl{itemi}
\item\label{ii} The commutator $[D, g]$ is bounded and relatively
 compact.\eqnnumlbl{itemii}
\end{enumerate}
\smallskip
In this case $t \mapsto (1-t) D +t g^{-1}D g = D+tg^{-1}[D,g]$ is a continuous
path in the graph topology and even in the stronger Riesz topology, see
\cite[Sec. 2]{L}. A systematic study of the spectral flow for such families of
operators, which are bounded perturbations of a fixed unbounded operator, was
given in \cite{bf98}.

For the sake of brevity we introduce the notation
\begin{equation}\label{EqSF}
\SF(D, g):= \SF ((D+tg^{-1}[D, g])_{0\le t \le1}).
\end{equation}

The homotopy invariance of the spectral flow \cite{blp}, Proposition 2.3
then implies
\begin{prop}\label{PropSF}
Let $D_s$, $0 \le s \le 1$, be a graph continuous family of selfadjoint
Fredholm operators. Furthermore, assume that $g_t$, $0 \le t \le 1$, is a
continuous family of unitary operators such that for any $0 \le s, t \le 1$,
the operator $g_t$ satisfies the conditions \eqref{itemi}, \eqref{itemii} above with
respect to $D_s$.  Assume also that $(s, t) \mapsto [D_s, g_t]$ is continuous
in the norm topology. Then
\begin{equation*}
  \SF(D_0, g_0) =\SF(D_1, g_1).
\end{equation*}
\end{prop}
\begin{proof} The map $(u, s, t) \mapsto D_s+ug_t^{-1}[D_s, g_t]\in \cf$ is
continuous in the graph topology.  Since $g_t^{-1} [D_s, g_t]$ is
$D_s$--compact, $D_s+ u g_t^{-1} [D_s,g_t]$ is indeed a Fredholm operator.
The closed path which is the concatenation of the paths
\begin{align*}
   f_1(u) &= D_u, \\
   f_2(u) &= D_1+ u g_1^{-1} [D_1,g_1], \\
   f_3(u) &= D_{1-u}+ g_{1-u}^{-1} [D_{1-u},g_{1-u}] & 0\le u\le 1, \\
          &= g_{1-u}^{-1} D_{1-u} g_{1-u}, \\
   f_4(u) &= D_0 + (1-u) g_0^{-1}[D_0, g_0],
\end{align*}
is homotopic to the constant path $D_0$ (since the cube $0\le u,s,t\le 1$
is contractible). By homotopy invariance, the total spectral flow over
this closed path vanishes. The path additivity of the spectral flow
yields
\[
\SF(f_1)+\SF(f_2)+\SF(f_3)+\SF(f_4)=0.
\]
From the definition \eqref{DefSF} \mpar{such labels
are a nuisance!!!} we infer
\[
\SF(f_1)+\SF(f_3)=0.
\]
Furthermore, $\SF(f_2)=\SF(D_1,g_1)$ and $\SF(f_4)=-\SF(D_0,g_0)$. The
proposition now follows.
\end{proof}

The following result of E.~Getzler \cite{G} computes the spectral flow in heat
equation terms.
\begin{theorem}[Getzler]\label{ThmGetzler}
Let $D$ be a selfadjoint operator on a Hilbert space
$H$ such that $e^{-\eps D^2}$ is trace class for any $\eps >0$. Furthermore,
let  $g \in U(H)$ be such that \eqref{itemi} and \eqref{itemii} are fulfilled. Then  for any $\eps >0$
 \begin{equation}\label{Getzler}
\SF(D, g)
   =\sqrt{\frac{\eps}{\pi}} \int \limits_0^1 \Tr\bigl( g^{-1}[D, g] e^{-\eps
   D(u)^2}\bigr)\ du,
\end{equation}
where $D(u)= (1-u)D +u g^{-1}D g = D+u g^{-1} [D, g]$.
\end{theorem}
\begin{remark}
 1. The formula \eqref{Getzler} is derived in \cite{G} under the assumption
that $D$ is invertible. It is true without this assumption as well. Indeed,
assuming this result for invertible $D$, consider for general $D$ the
function    $\mathbb{R} \ni t \mapsto f(t)=\sqrt{\frac{\eps}{\pi}} \int
\limits_0^1 \Tr\bigl( g^{-1}[D_t, g] e^{-\eps D_t(u)^2}\bigr)\ du$, where $D_t = D-t$.
This function  takes the value $\SF(D-t, g)= \SF(D, g)$, cf. Prop.
\ref{PropSF}, when $t$ is in the
complement of the spectrum of $D$. On the other hand    $ e^{-\eps
D_t(u)^2}$ is continuous as a function of $t$ and bounded as a function of  $u
\in [0, 1]$ in the space of trace class operators. Therefore $f(t)$ is
continuous and hence constant.

2. Getzler's formula has been generalized considerably. See
\cite{Carey2009} and the references therein.
\end{remark}
\section{The main results}\label{SecMain} 

\subsection{Boundary conditions for Dirac operators}\label{SSecBCDO} 

We consider an even dimensional compact Riemannian manifold $(M,g)$ with
boundary $\pl M$.  Furthermore, let $D$ be a Dirac operator acting on the
$\mathbb{Z}_2$--graded Clifford module $\cE = \cE^0 \oplus
\cE^1$ over $M$; that is $\cE$ is a module over the Clifford
algebra $\C(M)=C\ell(TM)$ and $D$ is an odd parity (with respect to the
$\mathbb{Z}_2$--grading) first order selfadjoint elliptic differential
operator satisfying
\begin{equation}
    D(f\cdot s) =\sfc(df)\cdot s + f\cdot Ds,\quad\text{for } f\in C^\infty(M),
    s\in \Gamma^\infty(\cE).
\end{equation}
Let $\bm{n}$ be the (inward) normal vector field at the boundary. Then
Clifford multiplication by $\bm{n}$ induces an isomorphism $J= \sfc(\bm{n})
\colon \cE^0|_{\pl M} \to \cE^1|_{\pl M}$.  Both
$\cE^0|_{\pl M}$ and $\cE^1|_{\pl M}$ become modules
over the Clifford algebra $\C(\pl M)$ by putting for $v \in
\Gamma^\infty(T\pl M)$
\begin{equation}
 \sfc_\cE(v) = \sfc(\bm{n}) \circ \begin{bmatrix}
                              \sfc_{\cE^0}(v) &0 \\
                                  0           &\sfc_{\cE^1}(v)
                            \end{bmatrix}.
\end{equation}
We denote by $D^\pl$ the corresponding Dirac operator on $\cE^0|_{\pl M}$.

We define \emph{local boundary conditions} for $D$ as follows:
For an invertible element  $F \in \End_{\C(\pl M)}(\cE^0|_{\pl M})$
put
\begin{equation}\label{EqBC}
\Dom(D_F) =
\bigsetdef{ f^0 \oplus f^1 \in \Gamma^\infty(\cE)}{%
   f^1|_{\pl M} =  J F f^0|_{\pl M} }.
\end{equation}
\newcommand{\ShapLop}{{\v S}apiro-Lopatinski{\v i}}
 \begin{prop}\label{PropF}
  \begin{enumerate}[\upshape a)]
  \item $D_F$ with the boundary condition \eqref{EqBC} is locally elliptic
   in the sense of \ShapLop.
  \item If $F^*=F$ then $D_F$ is selfadjoint.
  \item \label{PropFc} If $M$ is connected and if $F\ge 0$ or $F \le 0$ then $D_F$ is invertible.
\end{enumerate}
\end{prop}
\begin{proof}
Note first that, by the Clifford relations, $J$ is unitary and $J^2=-I$.

a) By \cite[Sec. 4.2]{blz} the \ShapLop\ condition is satisfied
if $F=J^tJ F$ is $>0$ or $<0$. However, since
$F\in\End_{\C(\pl M)}(\cE^0|_{\pl M})$ the splitting of
$\cE^0|_{\pl M}$ into the positive/negative spectral subbundles of $F$
gives a splitting of the boundary value problem in a collar neighborhood
of the boundary into a direct sum of two problems to which
\cite[Prop. 4.3]{blz} applies. For more details, see
\cite[Sec. 4.2]{blz}.

b) Given $s\in\Dom(D_F)$ and $t\in\Gamma^\infty(\cE)$. Green's formula
gives
\begin{equation}\label{EqGreen}
\langle Ds, t \rangle -\langle s, Dt \rangle
    = -\langle Js|_{\pl M}, t|_{\pl M}  \rangle
    = \langle s^0|_{\pl M}, J t^1|_{\pl M} + F^* t^0|_{\pl M}\rangle.
\end{equation}
Hence $t\in\Dom(D_F^*)$ iff $t^1|_{\pl M}= J F^* t^0|_{\pl M}$. Together
with a) it implies that $D_F=D_F^*$ iff $F=F^*$.

c) It is enough to show that $\Ker D_F =0$. Assume
$D_F f=0$, $f = f^0\oplus f^1$.
Choosing in Green's formula \eqref{EqGreen} $s= f^0\oplus 0$, $t=0\oplus f^1$
and using $f^1|_{\pl M} =  J F f^0|_{\pl M}$ we have
\begin{equation*}
   0 = \langle J  f^0|_{\pl M}, J F f^0|_{\pl M} \rangle
     = \langle   f^0|_{\pl M},    F f^0|_{\pl M} \rangle.
\end{equation*}
Positivity or negativity of $F$ implies that $f^0|_{\pl M}=0$,
and hence $f^1|_{\pl M}=0$. The Unique Continuation Property (\UCP)
for Dirac operators \cite[Sec. 9]{bw} now implies that $f=0$.
\end{proof}
\begin{remark}\label{rem2}
The statement of this Proposition with the same proof holds for more general
first order elliptic differential operators (for part c) one needs to assume
\UCP). In particular it holds for the operators of the
form $D + A$ where $A$ is an odd, selfadjoint endomorphism of $\cE$.
\end{remark}
\subsection{Main results}\label{SSMainResults} 

We continue in the notations of the previous subsection. Let $F \in
\End_{\C(\pl M)}(\cE^0|_{\pl M})$ be a selfadjoint
invertible element defining boundary conditions \eqref{EqBC}
for $D$. $D_F$ then is a selfadjoint realization of an elliptic boundary
value problem.

Let $P^+$ (respectively $P^-$) be the projection onto the positive 
(resp.~negative) eigenbundles of $F$; $P^\pm \in \End_{\C(\pl
M)}(\cE^0|_{\pl M})$.  Let $\cF^\pm = P^\pm(\cE^0)$. Then
$\cF^+\oplus \cF^- = \cE^0|_{\pl M}$
and $\cF^\pm$ are Clifford submodules of
$\cE^0|_{\pl M}$. Let $B^\pm$ be the corresponding Dirac
operators on  $\cF^\pm$. With $D^\pl$  the Dirac operator on $\cE^0|_{\pl M}$ \mpar{$\clubsuit$}
\begin{equation} B^\pm = P^\pm \circ D^\pl \circ P^\pm
\end{equation}
Set $B=B^+ \oplus B^-$, then
$D^\pl- B$ is a $0$-order operator.

Let $g \in C^\infty(M, U(n))$. Since the boundary condition of $D_F$ is local,
multiplication by $g$ preserves the domain of $D_F\otimes \id$ and hence $g$
satisfies (i) and (ii) on page \pageref{i} and thus the spectral flow
$\SF(D_F\otimes \id,g)$, cf. \eqref{EqSF}, is well--defined.  The analogous
problem for nonlocal boundary conditions is more subtle since then
multiplication by $g$ does not preserve the domain, see \cite{dz}.

Here, $D_F \otimes \id$ is acting on $\cE\otimes \mathbb{C}^N$.
By slight abuse of notation we will write again $D_F$ for
$D_F \otimes \id$ and use the analogous  convention for $B^\pm$, etc. \mpar{$\clubsuit$}

The following Theorem expresses $\SF(D_F,g)$ in terms of boundary data:
\begin{theorem}\label{ThmSF1}
\mpar{$\clubsuit$} Let $D_F$ be the Dirac operator on $M$ with the boundary conditions \eqref{EqBC} and let $B^\pm$
be the operators on $\pl M$ defined in \eqref{DefB}. Then
\begin{equation*}\label{DefB}
 \SF(D_F, g) = \frac{1}{2}\bigl( \SF(B^+, g|_{\pl M})-\SF(B^-, g|_{\pl
 M})\bigr).
\end{equation*}
\end{theorem}
\begin{proof}
We first show that we can reduce the computation of the spectral flow to the
situation where all structures are product near the boundary.  Our proof
relies on the results from \cite{blz}, and the argument is very similar to the
one in \cite{pro}. We then apply Getzler's formula \eqref{Getzler}. This
application uses standard calculations of the heat kernel on
the cylinder, which we relegate to the Section \ref{SecHK}.

Let $r$, $h$ be the Riemannian metric on $M$ and Hermitian metric on $\cE$
respectively.  Fix a collar neighborhood $N_0$ of the boundary, and fix
an identification $N_0\cong \pl M \times[0, \delta)$ using the exponential
mapping, so that $r( \frac{\pl}{\pl x}, \frac{\pl}{\pl x})=1$, $x \in [0,
\delta)$.  Let $r_1$ denote the  Riemannian metric on $M$ which is product on
$N_0$ and coincides with $r$ on $\pl M$.

We fix an isomorphism $\cE^0  \cong \pi^* \cE^0|_{\pl M}$ and let $h_1$ denote
the  Hermitian metric on $\cE^0$ which makes this isomorphism an isometry
over $N_0$.  Extend $h_1$ to $\cE^1$ so that $J= \sfc(\frac{\pl}{\pl x})$ is an
isometry.

One can construct, see \cite {blz}, an invertible even parity section 
$\Psi \in \Gamma^\infty(M,\End(\cE))$, which induces a unitary isomorphism
$L^2(M, \cE, r, h) \to L^2(M, \cE, r_1, h_1)$ and such that $\Psi|_{\pl M}=\id$.

We can use $J= \sfc(\frac{\pl}{\pl x})$ to identify $\cE$ with
$E := \cE^0\oplus \cE^0$ near the boundary via the map
$\cE \ni f^0 \oplus f^1 \mapsto f^0 \oplus (-J f^1) \in E$.
Under these identifications $\Psi \circ D \circ \Psi^{-1}$
can be written on $\pl M \times [0, \delta)$ as
\begin{equation}
 \begin{split}
  \Psi \circ D \circ \Psi^{-1}
    &= \begin{bmatrix} 0 &-1\\ 1 & 0 \end{bmatrix}
      \left(\frac{d}{dx} \otimes \id +\begin{bmatrix} D_x^\pl& 0 \\ 0
       &-D_x^\pl \end{bmatrix}\right) \\
    &= \begin{bmatrix}0 &-\frac{d}{dx}+D_x^\pl \\
          \frac{d}{dx}+D_x^\pl & 0 \end{bmatrix},
 \end{split}
\end{equation}
where $D_x^\pl$, $x \in [0, \delta)$, is a
smooth family of elliptic first order operators on $\pl M$ acting on the
sections of $\cE^0$ with $D_0^\pl =D^\pl$.

The boundary conditions \eqref{EqBC} for $\Psi \circ D \circ \Psi^{-1}$
under this identification take the form:
\begin{equation*}
 \bigsetdef{f^0 \oplus f^1 \in \Gamma^\infty(E)}{f^1|_{\pl M} =   F f^0|_{\pl M}}.
\end{equation*}

We can then construct a smooth family of elliptic self-adjoint boundary value problems
connecting the pair $(\Psi \circ D \circ \Psi^{-1}, F)$   to the pair of the form
$(\tD, {\tF})$ where
\begin{equation}
\tD  = \begin{bmatrix}0 &-\frac{d}{dx}+B \\ \frac{d}{dx}+B &0 \end{bmatrix},
\end{equation}
on $N_0$ and ${\tF} = P^+ -P^-$. Notice that by the Theorem 7.16 of \cite{blz} the corresponding
family of operators is continuous in the graph topology. We can also smoothly deform $g$ in the class
of smooth $U(n)$-valued maps to such an element $\tg$ which is independent of $x$ near $\pl M$ and such that
$\tg|_{\pl M} = g$.

Then $\SF(D_F, g) = \SF ((\Psi \circ D_F \circ \Psi^{-1}), g) =
\SF (\tD_{{\tF}}, g) = \SF( \tD_{{\tF}}, \tg)$
due to Proposition \ref{PropSF}.

Let $\mu_0, \lambda_0 \in C_0^\infty(N_0)$, $0 \le \mu_0, \lambda_0 \le 1$ be
cut-off functions such that $\mu_0$ is equal to $1$ near $\pl M$ and $\lambda_0$ 
is equal to $1$ in a neighborhood of $\supp \mu_0$. Let $\mu_1 =1 -\mu_0$ and 
$\lambda_1 \in C_0^\infty(M\setminus \pl M) $ be equal to $1$ in a neighborhood 
of $\supp \mu_1$. Then, using $\lambda_0\mu_0 = \mu_0$, $\lambda_1\mu_1 = \mu_1$,
we obtain for any trace class operator $K$ with smooth kernel on $M\times M$
the identity
\begin{equation*}
  \begin{split}
    \Tr\bigl( K \bigr) & = \Tr\bigl( \mu_0 K\bigr) + \Tr\bigl( \mu_1 K \bigr),\\
       & =  \Tr\bigl(\lambda_0 \mu_0 K\bigr) + \Tr\bigl(\lambda_1 \mu_1 K \bigr),\\
       & = \Tr\bigl( \lambda_0 K \mu_0 \bigr) + \Tr\bigl( \lambda_1 K \mu_1 \bigr).\\
  \end{split}
\end{equation*}
We emphasize that this is an exact formula. Specializing to 
$K= \tg^{-1}[\tD_{\tF}, \tg] e^{-\eps \tD_{\tF}(u)^2}$ we find 
\begin{multline*}
\sqrt{\frac{\eps}{\pi}} \int \limits_0^1
   \Tr\bigl( \tg^{-1}[\tD_{\tF}, \tg] e^{-\eps \tD_{\tF}(u)^2}\bigr) \ du \\
   =
\sqrt{\frac{\eps}{\pi}}
\int \limits_0^1 \Tr\bigl( \lambda_0 \tg^{-1}[\tD_{\tF}, \tg] e^{-\eps
\tD_{\tF}(u)^2} \mu_0\bigr)\ du\\
+ \sqrt{\frac{\eps}{\pi}} \int \limits_0^1 \Tr\bigl( \lambda_1 \tg^{-1}[\tD_{\tF}, \tg]
e^{-\eps \tD_{\tF}(u)^2} \mu_1\bigr)\ du.
\end{multline*}
By standard off-diagonal decay estimates for the heat kernel, cf. \cite[Sec. 3]{Lesch2013}, the first
term can asymptotically, i.~e. up to $O(\eps^\infty)$, be computed on the model cylinder.
This calculation is done in Section 4 below, Proposition \ref{PropCylinder}.
The second term is
$O(\eps^\infty)$. Indeed, all the local terms in the local asymptotic
expansion of $\Tr\bigl( \lambda_1 \tg^{-1}[\tD_{\tF}, \tg] e^{-\eps \tD_{\tF}(u)^2}
\mu_1\bigr)$ at the interior points vanish since
$\tg^{-1}[\tD, \tg] e^{-\eps \tD(u)^2}$ is an odd operator with respect to the
$\mathbb{Z}_2$ grading of $\cE$.
In conclusion applying Getzler's formula \eqref{Getzler} we get
\begin{equation*}\begin{split}
\SF(\tD_{\tF}, \tg)
 &=\sqrt{\frac{\eps}{\pi}} \int \limits_0^1 \Tr\bigl(
 \tg^{-1}[\tD_{\tF}, \tg] e^{-\eps \tD_{\tF}(u)^2}\bigr) \ du \\
 &=
   \frac{1}{2} \bigl( \SF(B^+, g) -\SF (B^-, g)\bigr) +O(\eps^\infty),
  \end{split}
\end{equation*}
when $\eps \to 0$. The statement follows.
\end{proof}

\begin{cor}[Cobordism invariance of spectral flow]
Let $g \in C^\infty(M, U(n))$. Then $\SF(D^\pl, g|_{\pl M})=0$.
\end{cor}
\begin{proof} For Dirac operators and connected $M$
 there is a short proof which makes use of the Unique Continuation Property of
 Dirac operators: Set $F=1$. Then $\cF^+= \cE^0|_{\pl M}$, $B^+ =
 D^\pl$ and $\cF^-=0$. By Proposition \ref{PropF}, \ref{PropFc}) and Remark \ref{rem2} all
 the operators $D_F(u) = D_F+ u g^{-1}[D_F, g]$ are invertible and hence
 Theorem \ref{ThmSF1} implies that $\frac{1}{2} \SF(D^\pl, g|_{\pl
 M})=\SF(D_F, g) =0$.

 We give a second proof which does not rely on \UCP\ nor on the connectedness
 of $M$ and hence is valid for general first order $\Z_2$--graded elliptic
 operators. Let $F>0$ or $<0$ and let
 \begin{equation}\label{EqGradingOperator}
  \gamma=\begin{bmatrix} 1 & 0 \\ 0 & -1 \end{bmatrix}
\end{equation}
be the grading operator. By homotopy invariance Prop. \ref{PropSF}
we have $\SF(D_F,g)=\SF(D_F+\gamma,g)$. We claim that $D_F(u)+\gamma$ is
invertible for $0\le u\le 1$ and thus $\SF(D_F+\gamma,g)=0$.

Indeed as in the proof of Prop. \ref{PropF}, \ref{PropFc}) it suffices to show that
$\Ker \left(D_F(u)+\gamma\right)=0$. If $\bigl(D_F(u)+\gamma\bigr)f=0$ one shows as in the
said proof that $f|_{\pl M}=0$; note that $D(u)+\gamma$ satisfies the same
Green's formula \eqref{EqGreen} as $D$. But then
\begin{equation*}\begin{split}
    0 & = \scalar{(D_F(u)+\gamma)f}{(D_F(u)+\gamma)f} \\
      & = \scalar{D_F(u)f}{D_F(u)f}+
          \scalar{\gamma f}{\gamma f}+
          \scalar{D_F(u)f}{\gamma f}+\scalar{\gamma f}{D_F(u)f}.
\end{split}
\end{equation*}
Since $f|_{\pl M}=0$ we infer from Green's formula
\[
    \scalar{D_F(u)f}{\gamma f}+\scalar{\gamma f}{D_F(u)f}
    =
    \scalar{(D(u)\gamma+ \gamma D(u))f}{f}=0,
\]
since $D(u)$ is odd with respect to the grading $\gamma$.
Consequently,
\[ 0=  \scalar{D_F(u)f}{D_F(u)f}+\scalar{\gamma f}{\gamma f} \]
and hence $f=0$.
\end{proof}

\begin{cor}
\begin{equation*}
\SF(B^+, g|_{\pl M})+\SF(B^-, g|_{\pl M})=0.
\end{equation*}
\end{cor}
\begin{proof}
\begin{equation*}
\SF(B^+, g|_{\pl M})+\SF(B^-, g|_{\pl M})= \SF(B , g|_{\pl M})=\SF(D^\pl, g|_{\pl M})=0.
\end{equation*}
The second equality is due to the fact that $D^\pl- B$ is bounded.
\end{proof}

From this we obtain a different form of the Theorem \ref{ThmSF1}:
\begin{thm1} In the notations of the Theorem \ref{ThmSF1} we have 
\begin{equation*}
  \SF(D_F, g) =   \SF(B^+, g|_{\pl M})= -\SF(B^-, g|_{\pl M}).
\end{equation*}
\end{thm1}

Finally, the topological formula for the spectral flow of Dirac operators
yields, up to normalizing conventions
\begin{cor}
\begin{equation*}\SF(D_F, g)= \int \limits_{\pl M} \widehat{A}(\pl M) \Ch(g) \Ch (\cF^+/\mathcal{S}).\end{equation*}
\end{cor}
\section{The heat kernel on a half-cylinder}\label{SecHK} 

In this section we compute the heat kernel for the boundary value problem
\eqref{EqBC} on a half-cylinder where all structures are of product form.
This will allow us to evaluate Getzler's formula explicitly. This will
complete the proof of the main Theorem \ref{ThmSF1}.

Let $N$ be a compact Riemannian manifold and let $B$ be a generalized Dirac
operator in the sense of \cite{bgv}, acting on the sections of the
Hermitian vector bundle $E^0$. That is $B$ is a first order elliptic
differential operator such that the leading symbol of $B^2$ at $\xi\in TM$
equals $g(\xi,\xi)$, such operators are called generalized Laplace operators
in \cite{bgv}. Furthermore, let $E=E^0 \oplus E^0$
and $F \in \End(E^0)$, $F^*=F$. We assume that
\begin{equation}\label{EqFRelations}
F^2=\id \text{ and }[B, F]=0.
\end{equation}

We pull back all structures to the cylinder $N\times[0, \infty)$.
On this manifold consider the operator $D_F$ acting on the sections of $E$
given by the equation
\begin{equation}\label{DB}
D=   \begin{bmatrix}0 &-\frac{d}{dx}+B \\ \frac{d}{dx}+B &0 \end{bmatrix}
\end{equation}
  with the boundary conditions
\begin{equation}\label{EqBC2}
 \bigsetdef{ f^0 \oplus f^1 \in \Gamma^\infty(E)}{ f^1|_{\pl M} =   F f^0|_{\pl M} }.
\end{equation}
$f= f^0\oplus f^1\in \Dom(D_F^2)$ if and only if $f\in\Dom(D_F)^2$ and
$D_Ff\in\Dom(D_F)$; this is equivalent to
\begin{equation}\label{EqBCLaplace}
 f^1|_{x=0} = F f^0|_{x=0} \quad\text{ and }\quad  
      \Bigl(\frac{d}{dx}f^0 +F\frac{d}{dx} f^1\Bigr)\Big|_{x=0} =0 .
\end{equation}

Let $U$ be the unitary automorphism of $E$ given by
\begin{equation}\label{EqU}
  U = \frac{1}{\sqrt{2}}\begin{bmatrix} F &-1\\ 1 &F \end{bmatrix}, \quad
   U^{-1}=\frac{1}{\sqrt{2}} \begin{bmatrix} F &1\\ -1 &F \end{bmatrix}.
 \end{equation}
We note the relations
\begin{equation}\label{EqUComm}
 U D U\ii = \begin{bmatrix} -FB & -\frac{d}{dx}\\ \frac{d}{dx} & FB
 \end{bmatrix},\quad\text{and}\quad U D^2 U\ii =D^2.
\end{equation}
Thus $U$ commutes with the differential operator $D^2$; however it maps $\Dom(D_F)$
bijectively onto $\Dom(D_m^2)$. Here, $D_m^2$ is the square of the operator given by the
formula \eqref{DB} with the boundary conditions
\begin{equation*}
\Dom(D_m^2) =
\bigsetdef{f =f^0 \oplus f^1 \in \Gamma^\infty(E) }{ f^0|_{x=0} = 0,\ \frac{d}{dx} f^1|_{x=0}=0}.
\end{equation*}
Thus $U D_F^2 U^{-1} =D_m^2$.

Let $k^\pm_\eps$ be the integral operators on the sections of $E$ with the
kernels given respectively by
\begin{equation*}
 k^\pm_\eps (x, y) = \frac{1}{\sqrt{4 \pi \eps}} ( e^{-\frac{(x-y)^2}{4 \eps}} \pm
          e^{-\frac{(x+y)^2}{4 \eps}}) \id, \quad x, y \in [0, \infty).
\end{equation*}
Then
\begin{equation}\label{EqHK}
e^{-\eps D_m^2} = \begin{bmatrix} k^-_\eps & 0\\ 0 &k^+_\eps \end{bmatrix}
 e^{-\eps B^2}.
\end{equation}

We now recall the notation of the beginning of Section \ref{SSMainResults}.
Let $g \in C^\infty(N, U(n))$, extended to $N \times [0, \infty)$ as being
independent from $x$.  Choose cut-off functions $\lambda, \mu \in
C^\infty_0([0, \infty))$ which are equal to $1$ in a neighborhood of $0$.  As
before we put $D_F(u) = D_F+u g^{-1}[D_F, g]$, where by slight abuse of
notation, we write $D_F$ for $D_F \otimes \id$. Since $B$ commutes with $F$
it preserves the sections of the $\pm 1$ eigenbundles of $F$; the restriction
of $B$ to the $\pm 1$ eigenbundle of $F$ is denoted by $B^\pm$.
\begin{prop}\label{PropCylinder} With the notation introduced before we have
\begin{multline*}
    \sqrt{\frac{\eps}{\pi}} \int \limits_0^1
\Tr\bigl( \lambda g^{-1}[D_F, g] e^{-\eps D_F(u)^2} \mu\bigr)\ du \\
= \frac{1}{2} ( \SF(B^+, g) -\SF (B^-, g)) +O(\eps^\infty)
\end{multline*}
as $\eps \to 0$.
\end{prop}
\begin{proof}
Let $B(u)= B+ ug^{-1}[B, g]$, and similarly  $B^\pm(u)= B^\pm+ ug^{-1}[B^\pm, g]$.
By the conjugation invariance of the trace,  \eqref{EqHK}, \eqref{EqUComm} and
\[
U [D_F,g] U\ii = \begin{bmatrix} - F [B,g] & 0 \\ 0 &F[B,g] \end{bmatrix}
\]
we have
\begin{align}
 \Tr&\bigl( \gl g\ii [D_F,g] e^{-\eps D_F(u)^2} \mu  \bigr) \nonumber\\
 &=  \Tr\bigl( \gl g\ii \begin{bmatrix} -F [B,g] & 0 \\ 0 & F[B,g] \end{bmatrix}
     \begin{bmatrix} k_\eps^- & 0 \\ 0 & k_\eps^+ \end{bmatrix} e^{-\eps
      B(u)^2} \mu  \bigr) \nonumber\\
 &= \Tr\bigl( F g\ii [B,g] e^{-\eps B(u)^2} \gl (k_\eps^--k_\eps^+)\mu  \bigr) \nonumber\\
 &= \Tr_N\bigl( F g\ii [B,g] e^{-\eps B(u)^2}\bigr) \times
      \Tr_{[0,\infty)}\bigl(\gl
      (k_\eps^--k_\eps^+)\mu\bigr).\label{EqTraceProd}
\end{align}
Here $\Tr_N$ indicates that the trace is taken of an operator acting on the
bundle on the boundary and $\Tr_{[0,\infty)}$ indicates the trace of an
operator acting on $L^2([0,\infty))$. For the two traces in
\eqref{EqTraceProd} we find
\begin{equation*}
      \Tr_{[0,\infty)}\bigl(\gl (k_\eps^--k_\eps^+)\mu\bigr)
        =\int_0^\infty \frac{1}{\sqrt{\pi \eps}} e^{-x^2/\eps} \ dx
        = \frac 12 + O(\eps^\infty),
\end{equation*}
and
\begin{equation*}\begin{split}
\Tr_N&\bigl( F g\ii [B,g] e^{-\eps B(u)^2}\bigr)\\
     & =  \Tr\bigl( g^{-1} [B^+, g]e^{-\eps B^+(u)^2}\bigr)
        - \Tr\bigl( g^{-1} [B^-, g]e^{-\eps B^-(u)^2}\bigr).
\end{split}\end{equation*}
The result then follows from Theorem \ref{ThmGetzler}.
\end{proof}
\bibliographystyle{amsalpha-lmp}
\bibliography{biblio}

\end{document}